\documentclass{amsart}
\usepackage{graphicx,amssymb}
\newtheorem{thm}{Theorem}

\newtheorem{lem}[thm]{Lemma}
\newtheorem{prop}[thm]{Proposition}
\theoremstyle{definition}

\newtheorem{exam}[thm]{Example}
\theoremstyle{remark}

\numberwithin{equation}{section}

\begin{document}

\title[Ultraparacompactness and Ultranormality]{Ultraparacompactness and Ultranormality}
\author{Joseph Van Name}
\address{}
\email{jvanname@mail.usf.edu}


\begin{abstract}
In this note, we shall overview some results related to ultraparacompactness and ultranormality in the general topological
and point-free contexts. This note contains some standard results and counterexamples along with some results which are not that well known
and even of my new results.
\end{abstract}
\maketitle
\section{Ultraparacompactness and Ultranormality}
In this note, we shall give an overview of the notions of ultranormality and ultraparacompactness. Some of the results in this paper are standard
and well known, while some other results are not very well known and some are a part of my research.

A Hausdorff space $X$ is said to be \emph{ultraparacompact} if every open cover can be refined by a partition into clopen sets. A Hausdorff space is \emph{ultranormal} if and only if whenever $R,S$ are disjoint closed sets, there is a clopen set $C$ with $R\subseteq C$ and $S\subseteq C^{c}$. 
Clearly, the ultranormal spaces are precisely  the spaces with large inductive dimension zero. A Hausdorff space is said to be \emph{zero-dimensional} if it has a basis of clopen sets. Clearly every ultraparacompact space is paracompact, and every ultranormal space is normal. It is easy to see that every ultraparacompact space is ultranormal and every ultranormal space is zero-dimensional. However, we shall soon see that the converses fail to hold. 

Intuitively, ultraparacompactness, ultranormality, and zero-dimensionality are the zero-dimensional analogues of the notions of paracompactness, normality, and regularity. Many in the notions and results from general topology have analogous zero-dimensional notions and results. 
For example, as every compact regular space can be embedded in some cube $[0,1]^{I}$, every compact zero-dimensional space can be embedded into some zero-dimensional cube $\{0,1\}^{I}$. As every complete uniform space of non-measurable cardinality is realcompact, every complete uniform space of non-measurable cardinality generated by equivalence relations is $\mathbb{N}$-compact. Moreover, notions such as uniform spaces, proximity spaces, the Stone-Cech compactification, realcompactness, and the Hewitt realcompactification are analogous to the notions of non-Archimedean uniform spaces, zero-dimensional proximity spaces, the Banaschewski compactification, $\mathbb{N}$-compactness, and the $\mathbb{N}$-compactification. Many of the results in this paper are analogous to results that do not necessarily involve zero-dimensional spaces.

Often it is easier to produce results about zero-dimensional spaces, and the theory of zero-dimensional spaces is generally more elegant
than for topological spaces in general. This is because zero-dimensional spaces have enough clopen sets and clopen sets satisfy nice properties since they form a Boolean algebra. For example, the Banaschewski compactification of a zero-dimensional space $X$ is simply the set of all ultrafilters on the Boolean algebra of clopen subsets of $X$. On the other hand, to define the Stone-Cech compactification of a completely regular space $X$ you use ultrafilters on the lattice of zero sets of $X$. Clearly ultrafilters on Boolean algebras are nicer than ultrafilters on lattices.

\section{Results}
We shall start with some counterexamples that show that ultranormality and ultraparacompactness are distinct notions.
\begin{prop}
Give the space $\mathbb{R}$ the lower limit topology (that is the topology generated by the basis $\{[a,b)|a<b\}$. Then
$\mathbb{R}$ is an ultraparacompact space with this topology.
\end{prop}
\begin{proof}
We shall show that the interval $[0,\infty)$ is ultraparacompact. The fact that $\mathbb{R}$ is ultraparacompact will then follow since
$\mathbb{R}$ is homeomorphic to the sum of countably many copies of the interval $[0,\infty)$. In fact, the space $\mathbb{R}$ is homeomorphic
to $[0,\infty)$ itself. Let $\mathcal{U}$ be an open cover of $[0,\infty)$. Let $x_{0}=0$. If $\alpha$ is an ordinal and $\sup\{x_{\beta}|\beta<\alpha\}<\infty$, then let $x_{\alpha}$ be a real number such that $x_{\alpha}>\sup\{x_{\beta}|\beta<\alpha\}$ and $[\sup\{x_{\beta}|\beta<\alpha\},x_{\alpha})\subseteq U$ for some $U\in\mathcal{U}$. Then $\{[x_{\alpha},x_{\alpha+1})|\alpha\}$ is a partition of $[0,\infty)$ into clopen sets that refines $\mathcal{U}$. Thus, $[0,\infty)$ is ultraparacompact.
\end{proof}
To the contrary, it is well known that the product $\mathbb{R}\times\mathbb{R}$ where $\mathbb{R}$ is given the lower limit topology is not even normal. Therefore the product of ultraparacompact spaces may not even be normal.

\begin{lem}
Let $f:\omega_{1}\rightarrow\omega_{1}$ be a function such that $f(\alpha)<\alpha$ for each $\alpha<\omega_{1}$. Then there is some
$\beta<\omega_{1}$ such that $f(\alpha)=\beta$ for uncountably many $\beta<\omega_{1}$.
\end{lem}
\begin{proof}
Assume to the contrary that for each $\beta<\omega_{1}$, we have $f(\alpha)=\beta$ for only countably many ordinals $\alpha$.
Then for each $\beta<\omega_{1}$, we have $f(\alpha)\leq\beta$ for only countably many ordinals $\alpha$. Therefore, by induction, there
is an increasing sequence $(\alpha_{n})_{n}$ such that if $f(\beta)\geq\alpha_{n+1}$, then $f(\beta)\geq\alpha_{n}$ for all $n$.
However, if we let $\alpha=^{\lim}_{n\rightarrow\infty}\alpha_{n}$, then $\alpha\geq\alpha_{n+1}$ for all $n$, so
$f(\alpha)\geq\alpha_{n}$ for all $n$, hence $f(\alpha)\geq\alpha$. This contradicts the assumption that
$f(\alpha)<\alpha$ for each $\alpha<\omega_{1}$.
\end{proof}
\begin{prop}
The space $\omega_{1}$ of all countable ordinals with the order topology is ultranormal, but not paracompact.
\end{prop}
\begin{proof}
We shall first show that $\omega_{1}$ is ultranormal.
If $R,S$ are two disjoint closed subsets of $\omega_{1}$, then either $R$ or $S$ is bounded, so say $R$ is bounded by an ordinal $\alpha$. Then since $[0,\alpha]$ is compact and zero-dimensional, the set $[0,\alpha]$ is ultranormal. Therefore there is a clopen subset $C\subseteq[0,\alpha]$ with $R\subseteq C$ and $S\cap C=\emptyset$. However, the set $C$ is clopen in $\omega_{1}$ as well. Therefore $\omega_{1}$ is ultranormal.

We shall now show that $\omega_{1}$ is not paracompact. We shall show that the cover
$\{[0,\alpha)|\alpha<\omega_{1}\}$ does not have a locally finite open refinement since if $\mathcal{U}$ is an open refinement of $[0,\alpha)$, then for each $\alpha<\omega_{1}$ there is some $x_{\alpha}<\alpha$ where $(x_{\alpha},\alpha]\subseteq U$ for some $U\in\mathcal{U}$. However, since the mapping $\alpha\mapsto x_{\alpha}$ is regressive, there is an ordinal $\beta$ where $x_{\alpha}<\beta$ for uncountably many $\alpha$. Since $\mathcal{U}$ refines $\{[0,\alpha)|\alpha<\omega_{1}\}$, each $U\in\mathcal{U}$ is bounded, so the ordinal $\beta$ must be contained in uncountably many $U\in\mathcal{U}$.
\end{proof}

In the paper \cite{R}, Prabir Roy shows that certain space $\Delta$ is a complete metric space of cardinality continuum which is zero-dimensional, but not ultranormal, and hence not ultraparacompact. In fact, later in the paper \cite{N}, Peter Nyikos shows that this space is not even $\mathbb{N}$-compact (a space is \emph{$\mathbb{N}$-compact} if and only if it can be embedded as a closed subspace of a product $\mathbb{N}^{I}$ for some set $I$). This result strengthens Roy's result since every ultraparacompact space of non-measurable cardinality is $\mathbb{N}$-compact and the first measurable cardinal is terribly large if it even exists.

There are zero-dimensional locally compact spaces that are not ultranormal. The Tychonoff plank $X:=((\omega_{1}+1)\times(\omega+1))\setminus\{(\omega_{1},\omega)\}$ is zero-dimensional (even strongly zero-dimensional; i.e. $\beta X$ is zero-dimensional) but not ultranormal.

We shall now go over some results about ultraparacompact and ultranormal spaces. The following result gives many examples of ultraparacompact spaces.
\begin{prop}
Every Lindelof zero-dimensional space is ultraparacompact.
\end{prop}
\begin{proof}
Let $X$ be Lindelof and zero-dimensional. Let $\mathcal{U}$ be an open cover of $X$. Since $X$ is zero-dimensional,
the cover $\mathcal{U}$ is refinable by a clopen cover $\mathcal{C}$. Let $(C_{n})_{n}$ be a countable subcover of
$\mathcal{C}$ and let $R_{n}=C_{n}\setminus\bigcup_{m<n}C_{m}$ for all $n$. Then $\{R_{n}|n\in\mathbb{N}\}$ is a partition of
$X$ into clopen sets that refines $\mathcal{U}$.
\end{proof}
The following result shows that the locally compact ultraparacompact spaces have the simplest possible characterization.
\begin{prop}
Let $X$ be a locally compact zero-dimensional space. Then $X$ is ultraparacompact if and only if $X$ can be partitioned into a family of compact open sets.
\end{prop}
\begin{proof}
The direction $\leftarrow$ is fairly trivial. To prove $\rightarrow$ assume that $X$ is a locally compact zero-dimensional space. Then let $\mathcal{U}$ the collection of all open sets $U$ such that $\overline{U}$ is compact. Then since $X$ is locally compact, $\mathcal{U}$ is a cover for $X$. Therefore there is a partition $P$ of $X$ into clopen sets that refines $\mathcal{U}$. Clearly each $R\in P$ is a compact open subset of $X$. \end{proof}
Ultraparacompact spaces and ultranormal spaces both satisfy a generalized version of Tietze's extension theorem.
\begin{thm}
1. Let $C$ be a closed subset of an ultraparacompact space $X$ and let $Y$ be a complete metric space. Then every
continuous function $f:C\rightarrow Y$ can be extended to a continuous function $F:X\rightarrow Y$.

2. Let $C$ be a closed subset of an ultranormal space $X$ and let $Y$ be a complete separable metric space.
Then every continuous function $f:C\rightarrow Y$ can be extended to a continuous function $F:X\rightarrow Y$.
\end{thm}
\begin{proof}
See \cite{E}.
\end{proof}
Every closed subset of an ultranormal space is ultranormal, and every closed subset of an ultraparacompact space is ultraparacompact.
On the other hand, ultranormality and ultraparacompactness are not closed under taking arbitrary subspaces. The good news is that there are nice conditions that allow us to determine whether every subspace of an ultraparacompact space is ultraparacompact.
\begin{prop}
Let $X$ be a space. If each open subset of $X$ is ultraparacompact, then every subset of $X$ is ultraparacompact.
\end{prop}
\begin{proof}
Let $A\subseteq X$ and let $\{U_{\alpha}|\alpha\in\mathcal{A}\}$ be an open cover of $A$. Then extend each $U_{\alpha}$ to an open
$V_{\alpha}$ with $A\cap V_{\alpha}=U_{\alpha}$. Let $V=\bigcup_{\alpha\in\mathcal{A}}V_{\alpha}$. Then
$\{V_{\alpha}|\alpha\in\mathcal{A}\}$ is an open cover of $V$, so $\{V_{\alpha}|\alpha\in\mathcal{A}\}$ is refinable by some partition
$P$ of $V$ into clopen subsets of $V$. Therefore $\{R\cap A|R\in P\}$ is the required partition of $A$ into open sets.
\end{proof}

We will need a few definitions in order to state and obtain some nice characterizations of ultraparacompact spaces and ultranormal spaces.
A subset $Z$ of a space $X$ is said to be a \emph{zero set} if there is a continuous function $f:X\rightarrow[0,1]$ such that $Z=f^{-1}[\{0\}]$. A completely regular space $X$ is said to be \emph{strongly zero-dimensional} if whenever $Z_{1},Z_{2}$ are disjoint zero sets, there is a clopen set $C$ with $Z_{1}\subseteq C$ and $Z_{2}\subseteq C^{c}$. It is not too hard to show that a completely regular space $X$ is strongly zero-dimensional if and only if the Stone-Cech compactification $\beta X$ is zero-dimensional. We say that a cover $\mathcal{R} $ of a topological space $X$ is \emph{point-finite} if $\{R\in\mathcal{R}|x\in R\}$ is finite for each $x\in X$, and we say that $\mathcal{R}$ is \emph{locally finite} if each $x\in X$ has a neighborhood $U$ such that $\{R\in\mathcal{R}|U\cap R\neq\emptyset\}$ is finite. Recall that a Hausdorff space is paracompact iff every open cover has a locally finite open refinement.

A uniform space $(X,\mathcal{U})$ is said to be \emph{non-Archimedean} if $\mathcal{U}$ is generated by equivalence relations. In other words, for each $R\in\mathcal{U}$ there is an equivalence relation $E\in\mathcal{U}$ with $E\subseteq R$. Clearly, every non-Archimedean uniform space is zero-dimensional. If $(X,\mathcal{U})$ is a uniform space, then let $H(X)$ denote the set of all closed subsets of $X$. If $E\in\mathcal{U}$, then let $\hat{E}$ be the relation on $H(X)$ where $(C,D)\in\hat{E}$ if and only if $C\subseteq E[D]=\{y\in X|(x,y)\in E\textrm{for some}x\in D\}$ and $D\subseteq E[C]$. Then the system $\{\hat{E}|E\in\mathcal{U}\}$ generates a uniformity $\hat{\mathcal{U}}$ on $H(X)$ called the \emph{hyperspace} uniformity. A uniform space $(X,\mathcal{U})$ is said to be \emph{supercomplete} if the hyperspace $H(X)$ is a complete uniform space.

\begin{thm}
Let $X$ be a Hausdorff space. Then $X$ is normal if and only if whenever $(U_{\alpha})_{\alpha\in\mathcal{A}}$ is a point-finite open covering of $X$, there is an open covering $(V_{\alpha})_{\alpha\in\mathcal{A}}$ such that $\overline{V_{\alpha}}\subseteq U_{\alpha}$ for $\alpha\in\mathcal{A}$ and $V_{\alpha}\neq\emptyset$ whenever $U_{\alpha}\neq\emptyset$.
\end{thm}
To prove the above result, one first well orders the set $\mathcal{A}$, then one shrinks the sets $U_{\alpha}$ to sets $V_{\alpha}$ in such a way that you still cover your space $X$ at every point in the induction process. See the book \cite{D} for a proof of the above result.

\begin{thm} Let $X$ be a Hausdorff space. The following are equivalent.

\begin{enumerate}

\item $X$ is ultranormal.

\item Whenever $(U_{\alpha})_{\alpha\in\mathcal{A}}$ is a point-finite open cover of $X$, there is a clopen cover $(V_{\alpha})_{\alpha\in\mathcal{A}}$ such that $V_{\alpha}\subseteq U_{\alpha}$ for each $\alpha$ and $V_{\alpha}\neq\emptyset$ whenever $U_{\alpha}\neq\emptyset$.

\item If $(U_{\alpha})_{\alpha\in\mathcal{A}}$ is a locally-finite open cover of $X$, then there is system $(P_{\alpha})_{\alpha\in\mathcal{A}}$ of clopen sets such that $P_{\alpha}\subseteq U_{\alpha}$ for $\alpha\in\mathcal{A}$ and $P_{\alpha}\cap P_{\beta}=\emptyset$ whenever $\alpha,\beta\in\mathcal{A}$ and $\alpha\neq\beta$.

\item $X$ is normal and strongly zero-dimensional.

\end{enumerate}
\label{SK7y3arekl}
\end{thm}
\begin{proof}
$1\rightarrow 2$. Since $X$ is normal, there is an open cover $(W_{\alpha})_{\alpha\in\mathcal{A}}$ such that $\overline{W_{\alpha}}\subseteq U_{\alpha}$ for each $\alpha\in\mathcal{A}$ and $W_{\alpha}\neq\emptyset$ whenever $U_{\alpha}\neq\emptyset$. Since $X$ is ultranormal, for each $\alpha\in\mathcal{A}$, there is some clopen set $V_{\alpha}$ with $\overline{W_{\alpha}}\subseteq V_{\alpha}\subseteq U_{\alpha}$.

$2\rightarrow 3$. Now assume that $(U_{\alpha})_{\alpha\in\mathcal{A}}$ is a locally-finite open over of $X$. Let $(V_{\alpha})_{\alpha\in\mathcal{A}}$ be a clopen cover of $X$ such that $V_{\alpha}\subseteq U_{\alpha}$ for each $\alpha\in\mathcal{A}$. Well order the set $\mathcal{A}$. The family $(V_{\alpha})_{\alpha\in\mathcal{A}}$ is locally-finite, so since each $V_{\alpha}$ is closed, for each $\alpha\in\mathcal{A}$, the union $\bigcup_{\beta<\alpha}V_{\beta}$ is closed. Clearly $\bigcup_{\beta<\alpha}V_{\beta}$ is  open as well, so $\bigcup_{\beta<\alpha}V_{\beta}$ is clopen. Let
$P_{\alpha}=V_{\alpha}\setminus(\bigcup_{\beta<\alpha}V_{\beta})$. Then $(P_{\alpha})_{\alpha\in\mathcal{A}}$ is the required partition of $X$ into clopen sets.

$3\rightarrow 1,1\rightarrow 4$. This is fairly obvious.

$4\rightarrow 1$. This is a trivial consequence of Urysohn's lemma.
\end{proof}
\begin{thm} Let $X$ be a Hausdorff space. The following are equivalent.

\begin{enumerate}
\item $X$ is ultraparacompact.

\item $X$ is ultranormal and paracompact.

\item $X$ is strongly zero-dimensional and paracompact.

\item Every open cover of $X$ has a locally finite clopen refinement.

\item (Van Name) $X$ is zero-dimensional and satisfies the following property: let $I$ be an ideal on the Boolean algebra $\mathfrak{B}(X)$ of clopen subsets of $X$ such that $\bigcup I=X$ and if $P$ is a partition of $X$ into clopen sets, then $\bigcup(P\cap I)\in I$. Then $I=\mathfrak{B}(X)$.

\item (Van Name) $X$ is zero-dimensional and satisfies the following property: let $I$ be an ideal on $\mathfrak{B}(X)$ such that if
$P$ is a partition of $X$ into clopen sets, then $\bigcup(P\cap I)\in I$. Then there is an open set $O\subseteq X$ with
$I=\{R\in\mathfrak{B}(X)|R\subseteq O\}$.

\item $X$ has a compatible supercomplete non-Archimedean uniformity.
\end{enumerate}
\label{WERJ34qu}
\end{thm}
A metric $d$ on a set $X$ is said to be an ultrametric if $d$ satisfies the strong triangle inequality: $d(x,z)\leq Max(d(x,y),d(y,z))$, and a metric space $(X,d)$ is said to be an ultrametric space if $d$ is an ultrametric.
\begin{thm}
Every ultrametric space is ultraparacompact.
\end{thm}
\begin{proof}
It is well known that every metric space is paracompact. If $(X,d)$ is an ultrametric space and $R,S$ are disjoint closed subsets.
Let $f:X\rightarrow[0,1]$ be the function defined by $f(x)=\frac{d(C,x)}{d(C,x)+d(D,x)}$. Then $f$ is a continuous function where
$f=0$ on $C$, $f=1$ on $D$. Furthermore, if $x\not\in C\cup D$, then there is some $\epsilon>0$ with
$d(x,C)\geq\epsilon$ and $d(x,D)\geq\epsilon$. However, if $y\in B_{\epsilon}(x)$, then for each $c\in C$, we have
$d(y,c)\leq Max(d(x,y),d(x,c))=d(x,c)$. If $d(y,c)<d(x,c)$, then $d(x,c)\leq Max(d(y,c),d(x,y))<d(x,c)$, a contradiction, so
$d(y,c)=d(x,c)$. Therefore $d(y,C)=d(x,C)$ whenever $y\in B_{\epsilon}(x)$. Similarly, $d(y,C)=d(x,C)$ whenever
$y\in B_{\epsilon}(x)$. Therefore $f(y)=f(x)$ for $y\in B_{\epsilon}(x)$. In other words, the function $f$ is locally constant outside
$C\cup D$. We therefore conclude that $f^{-1}[0,\frac{1}{2})$ is a clopen set and clearly
$C\subseteq f^{-1}[0,\frac{1}{2})$ and $D\subseteq f^{-1}[\frac{1}{2},1]$. Therefore $X$ is ultranormal. Since
$X$ is ultranormal and paracompact, we conclude that $X$ is ultraparacompact.
\end{proof}

\begin{thm}
Let $X$ be an ultraparacompact space, and let $Y$ be a compact zero-dimensional space. Then $X\times Y$ is ultraparacompact
as well.
\label{EKJ379y4}
\end{thm}
\begin{proof}
The proof of this result uses standard compactness argument.
Let $\mathcal{U}=\{U_{\alpha}\times V_{\alpha}|\alpha\in\mathcal{A}\}$ be an open cover of $X\times Y$. Then for each
$x\in X$, there we have $Y$ be covered by $\{V_{\alpha}|x\in U_{\alpha}\}$. Therefore reduce this cover to a finite subcover
$V_{\alpha_{x,1}},...,V_{\alpha_{x,n_{x}}}$. Let $U_{x}=U_{\alpha_{x,1}},...,U_{\alpha_{x,n_{x}}}$ for each $x\in X$.
Then there is a partition $P$ of $X$ into clopen sets that refines $\{U_{x}|x\in X\}$. If $R\in P$, then there is some
$x_{R}$ with $R\subseteq U_{x_{R}}$. Furthermore, cover $V_{\alpha_{x,1}},...,V_{\alpha_{x,n_{x}}}$ is refinable by a partition
$P_{R}$ of $Y$ into clopen sets. It can easily be seen that $\bigcup_{R\in P}\{R\times S|S\in P_{R}\}$ is the required partition of
$X$ into clopen sets.
\end{proof}

The following well known characterization of paracompactness generalizes to ultraparacompactness.
\begin{thm}
Let $X$ be a completely regular space with compactification $C$. Then $X$ is paracompact if and only if $X\times C$ is normal
\end{thm}

\begin{thm}
Let $X$ be a zero-dimensional space with zero-dimensional compactification $C$. Then $X$ is ultraparacompact if and only if
$X\times C$ is ultranormal.
\end{thm}
\begin{proof}
$\rightarrow$. This follows from Theorem \ref{EKJ379y4}

$\leftarrow$ If $X\times C$ is ultranormal, then $X\times C$ is normal, so $X$ is paracompact. Furthermore, since
$X\times C$ is ultranormal, the space $X$ is ultranormal as well. Therefore, by Theorem \ref{WERJ34qu}, the space
$X$ is ultraparacompact.
\end{proof}

The paper \cite{AM} gives several characterizations of when a $P_{\kappa}$-space (a space where the intersection of less than $\kappa$ many open sets is open) is ultraparacompact.
We define a \emph{$P_{\kappa}$-space} to be a completely regular space where the intersection of less than $\kappa$ many open sets is open. A \emph{$P$-space} is a complete regular space where the intersection of countably many open sets is open.
If $\kappa$ is a singular cardinal, then it is easy to show that every $P_{\kappa}$-space is a $P_{\kappa^{+}}$-space, so it suffices to restrict the study of $P_{\kappa}$-spaces to when $\kappa$ is a regular cardinal. We say that a collection $\mathcal{O}$ of subsets of a space $X$ is \emph{locally
$\kappa$-small} if each $x\in X$ has an open neighborhood $U$ such that $|\{O\in\mathcal{O}|O\cap U\neq\emptyset\}|<\kappa$.
\begin{thm}
Let $\kappa$ be an uncountable regular cardinal and let $X$ be a $P_{\kappa}$-space. Then the following are equivalent.
\begin{enumerate}
\item Every open cover is refinable by a locally $\kappa$-small open cover.

\item Every open cover is refinable by a locally $\kappa$-small cover consisting of sets that are not necessarily open.

\item Every open cover is refinable by a locally $\kappa$-small closed cover.

\item Every open cover is refinable by a locally $\kappa$-small clopen cover.

\item $X$ is paracompact.

\item $X$ is ultraparacompact.

\item Every open cover is refinable by an open cover that can be partitioned into at most $\kappa$ many locally $\kappa$-small families.
\end{enumerate}
\end{thm}
\begin{proof}
See \cite{AM}.
\end{proof}
\begin{thm}
Let $G$ be a topological group whose underlying topology is an ultraparacompact $P$-space. Then every open cover of $G$ is refinable by a partition
of $G$ into cosets of open subgroups of $G$.
\end{thm}

\section{The Point-Free Context}
The remaining results on strong zero-dimensionality, ultranormality, and ultraparacompactness are all from my research.
The notions of ultranormality, zero-dimensionality, and ultraparacompactness make sense in a point-free context, and by a generalization of Stone duality, the notions of ultranormality, zero-dimensionality, and ultraparacompactness translate nicely to certain kinds of Boolean algebras with extra structure.

A \emph{frame} is a complete lattice that satisfies the following infinite distributivity law
\[x\wedge\bigvee_{i\in I}y_{i}=\bigvee_{i\in I}(x\wedge y_{i}).\]
If $X$ is a topological space, then the collection of all open subsets of $X$ forms a frame.
Frames generalize the notion of a topological space, and frames are the central object of study in point-free topology. For Hausdorff spaces, no information is lost simply by considering the lattice of open sets of the topological space. More specifically, if $X,Y$ are Hausdorff spaces and the lattices of open sets of $X$ and $Y$ respectively are isomorphic, then $X$ and $Y$ are themselves isomorphic. Furthermore, many notions and theorems from general topology can be generalized to the point-free context. Moreover, the notion of a frame is very interesting from a purely lattice theoretic perspective without even looking at the topological perspective. The reader is referred to the excellent new book \cite{P}
for more information on point-free topology.

For notation, if $X$ is a poset and $R,S\subseteq X$, then $R$ refines $S$ (written $R\preceq S$) if for each $r\in R$ there is some $s\in S$ with $r\leq s$. If $x\in X$, then define $\downarrow x:=\{y\in X|y\leq x\}$.

A \emph{Boolean admissibility system} is a pair $(B,\mathcal{A})$ such that $B$ is a Boolean algebra and $\mathcal{A}$ is a collection of subsets of $B$ with least upper bounds such that

i. $\mathcal{A}$ contains all finite subsets of $B$,

ii. if $R\in\mathcal{A},S\subseteq\downarrow\bigvee R,R\preceq S$, then $S\in\mathcal{A}$ as well,

iii. if $R\in\mathcal{A}$ and $R_{r}\in\mathcal{A},\bigvee R_{r}=r$ for $r\in R$, then
$\bigcup_{r\in R}R_{r}\in\mathcal{A}$,

iv. if $R\in\mathcal{A}$, then $\{a\wedge r|r\in R\}\in\mathcal{A}$ as well.

Intuitively, a Boolean admissibility system is Boolean algebra along with a notion of which least upper bounds are important and which least upper bounds are not important. 
\begin{exam}
If $B$ is a Boolean algebra, and $\mathcal{A}$ is the collection of all subsets of $B$ with least upper bounds, then $(B,\mathcal{A})$ is a Boolean admissibility system.
\end{exam}
\begin{exam}
If $A$ is a Boolean subalgebra of $B$, and $\mathcal{A}$ is the collection of all subsets of $R$ where the least upper bound $\bigvee^{B}R$ exists in $B$ and $\bigvee^{B}R\in A$. Then $(A,\mathcal{A})$ is a Boolean admissibility system.
\end{exam}
A Boolean admissibility system $(B,\mathcal{A})$ is said to be subcomplete if whenever $R,S\subseteq B$ and $R\cup S\in\mathcal{A}$ and $r\wedge s=\emptyset$ whenever $r\in R,s\in S$, then $R\in\mathcal{A}$ and $S\in\mathcal{A}$.
\begin{exam}
If $(B,\mathcal{A})$ is a Boolean admissibility system and $B$ is a complete Boolean algebra, then the Boolean admissibility system
$(B,\mathcal{A})$ is subcomplete.
\end{exam}

If $L$ is a frame, then an element $x\in L$ is said to be \emph{complemented} if there is an element $y$ such that $x\wedge y=0$ and $x\vee y=1$. The element $y$ is said to be the complement of $y$ and one can easily show that the element $y$ is unique. The notion of a complemented element is the point-free generalization of the notion of a clopen set. Let $\mathfrak{B}(L)$ denote the set of complemented elements in $L$. Then $\mathfrak{B}(L)$ is a sublattice of $L$. In fact, $\mathfrak{B}(L)$ is a Boolean lattice. A frame $L$ is said to be \emph{zero-dimensional} if $x=\bigvee\{y\in\mathfrak{B}(L)|y\leq x\}$ for each $x\in L$. 

A \emph{Boolean based frame} is a pair $(L,B)$ where $L$ is a frame and $B$ is a Boolean subalgebra of $\mathfrak{B}(L)$ such that $x=\bigvee\{b\in\mathfrak{B}(L)|b\leq x\}$ for each $x\in L$. Clearly every Boolean based frame is zero-dimensional.
A zero-dimensional frame $L$ is said to be \emph{ultranormal} if whenever $a\vee b=1$, there is a complemented element $c\in L$ such that $c\leq a$ and $c'\leq b$.
A \emph{cover} of a frame $L$ is a subset $C\subseteq L$ with $\bigvee C=1$. A \emph{partition} of a frame $L$ is a subset $p\subseteq L\setminus\{0\}$ with $\bigvee p=1$ and where $a\wedge b=0$ whenever $a,b\in p,a\neq b$.
A zero-dimensional frame $L$ is said to be \emph{ultraparacompact} if whenever $C$ is a cover of $L$ there is a partition $p$ that refines $C$.
It can easily be seen that a Hausdorff space $(X,\tau)$ is zero-dimensional, ultranormal, or ultraparacompact respectively if and only if
$\tau$ is a zero-dimensional, ultranormal, or ultraparacompact respectively frame.

\begin{exam}
Every complete Boolean algebra is an ultraparacompact frame.
\end{exam}

There is a duality between the category of Boolean admissibility systems and Boolean based frames. If $(B,\mathcal{A})$ is a Boolean admissibility system, then let $C_{\mathcal{A}}$ be the collection of all ideals $I\subseteq B$ such that if $R\in\mathcal{A}$ and $R\subseteq I$, then $\bigvee R\in I$ as well.

If $(B,\mathcal{A})$ is a Boolean admissibility system, then $\mathcal{V}(B,\mathcal{A}):=(C_{\mathcal{A}},\{\downarrow b|b\in B\})$ is a Boolean based frame. Similarly, if $(L,A)$ is a Boolean based frame, then $\mathcal{W}(L,A):=(A,\{R\subseteq A|\bigvee^{L}R\in A\})$ is a Boolean admissibility system. Furthermore, these correspondences give an equivalence between the category of Boolean admissibility systems and the category of Boolean based frames. Hence, we obtain a type of Stone-duality for zero-dimensional frames. If $(L,A)$ is a Boolean based frame, then $A=\mathfrak{B}(L)$ if and only if $\mathcal{W}(L,A)$ is subcomplete. Since $(L,\mathfrak{B}(L))$ is a Boolean based frame iff $L$ is a zero-dimensional frame, we conclude that the category of zero-dimensional frames is equivalent to the category of subcomplete Boolean admissibility systems.

\begin{thm}
If $(B,\mathcal{A})$ is a Boolean admissibility system, then $
\mathcal{V}(B,\mathcal{A})=(C_{\mathcal{A}},\{\downarrow b|b\in B\})$ is ultranormal with 
$\{\downarrow b|b\in B\}=\mathfrak{B}(C_{\mathcal{A}})$ if and only if whenever $I,J\in\mathcal{C}_{\mathcal{A}}$, then $\{a\vee b|a\in I,b\in J\}\in\mathcal{C}_{\mathcal{A}}$ as well.
\end{thm}
Hence, ultranormality simply means that the join of finitely many ideals in $\mathcal{C}_{\mathcal{A}}$ is an ideal in $\mathcal{C}_{\mathcal{A}}$.

\begin{thm}
If $(B,\mathcal{A})$ is a subcomplete Boolean admissibility system, then $\mathcal{V}(B,\mathcal{A})$ is ultraparacompact if and only if whenever $R\in\mathcal{A}$ there is some $S\in\mathcal{A}$ with $S\preceq R,\bigvee S=\bigvee R$ and $a\wedge b=0$ whenever $a,b\in S$ and $a\neq b$.
\end{thm}
Hence, the ultraparacompactness of frames translates into a version of ultraparacompactness in the dual subcomplete Boolean admissibility systems.

A \emph{partition} of a Boolean algebra $B$ is a subset $p\subseteq B\setminus\{0\}$ with $\bigvee p=1$ such that if
$a,b\in p,a\neq b$, then $a\wedge b=0$. The collection of partitions of a Boolean algebra forms a meet-semilattice under the refinement ordering
$\preceq$ with $p\wedge q=\{a\wedge b|a\in p,b\in q\}\setminus\{0\}$.
A \emph{Boolean partition algebra} is a pair $(B,F)$ where $B$ is a Boolean algebra and $F$ is a filter on the meet-semilattice of partitions of
$B$ such that $B=\{0\}\cup\bigcup F$. If $(A,F),(B,G)$ are Boolean partition algebras, then a \emph{partition homomorphism} from
$(A,F)$ to $(B,G)$ is a Boolean algebra homomorphism $\phi:(A,F)\rightarrow(B,G)$ such that $\phi[p]^{+}\in G$ whenever
$p\in F$.

A partition $p$ of a Boolean algebra $B$ is said to be \emph{subcomplete} if whenever $R\subseteq p$, the least upper bound
$\bigvee R$ exists. A Boolean partition algebra is said to be \emph{subcomplete} if each $p\in F$ is subcomplete.
A Boolean partition algebra $(B,F)$ is said to be \emph{locally refinable} if whenever
$p\in F,|p|>1$ and $p_{a}$ is a partition of $a$ with $p_{a}\cup\{a'\}\in F$ for $a\in p$, then $\bigcup_{a\in p}p_{a}\in F$ as well.

\begin{exam}
If $B$ is a Boolean algebra and $F$ is a filter on the lattice of partitions of $B$, then $\{0\}\cup\bigcup F$ is a Boolean algebra.
Furthermore, $(\{0\}\cup\bigcup F,F)$ is a Boolean partition algebra. If each $p\in F$ is a subcomplete partition of $B$, then
$(\{0\}\cup\bigcup F,F)$ is a subcomplete Boolean partition algebra.
\end{exam}

\begin{thm}
The category of subcomplete locally refinable Boolean partition algebras is equivalent to the category of ultraparacompact frames.
\end{thm}

\end{document}